\documentclass[9pt]{article}%
\usepackage{bbm}
\usepackage{epsfig}
\usepackage{graphics,graphicx,amssymb,amsmath,verbatim}
\usepackage{amssymb}
\usepackage{mathrsfs}
\usepackage{amsfonts}
\usepackage{amsmath}
\usepackage{graphicx}
\usepackage{easybmat}%
\providecommand{\U}[1]{\protect \rule{.1in}{.1in}}
\newtheorem{theorem}{Theorem}

\newtheorem{corollary}{Corollary}

\newtheorem{proposition}{Proposition}
\newtheorem{remark}{Remark}

\newenvironment{proof}[1][Proof]{\noindent \textbf{#1.} }{\  \rule{0.5em}{0.5em}}
\allowdisplaybreaks[4]
\begin{document}

\title{On Stability of the Linearized Spacecraft Attitude Control System}
\author{Bin Zhou\thanks{Center for Control Theory and Guidance Technology, Harbin
Institute of Technology, P.O. Box 416, Harbin, 150001, China. Email:
\texttt{binzhou@hit.edu.cn; binzhoulee@163.com.}}}
\date{}
\maketitle

\begin{abstract}
This note is concerned with the stability and stabilization of the linearized
spacecraft attitude control system. Necessary and sufficient conditions are
respectively provided to guarantee that the considered systems are
polynomially stable and stable in the Lyapunov sense. These two classes of
conditions guarantee that the linearized spacecraft attitude control system
can be respectively stabilized semi-globally and globally by saturated linear
state feedback. \vspace{0.3cm}

\textbf{Keywords:} Spacecraft Attitude control; Stability and stabilization;
Lyapunov stable; Saturated feedback.

\end{abstract}

\section{Introduction}

Attitude stabilization of spacecraft is a very important problem since it
helps the spacecraft to maintain a certain prescribed attitude in the space
during their useful life \cite{gerlach65ssr}, \cite{la04auto},
\cite{psiaki01jgcd}, \cite{rh11icrast}, \cite{tsiotras}, \cite{wertz78book}.
Attitude stabilization controllers design can be based on either the nonlinear
dynamics or the linearized dynamics. If the controllers are designed based on
the nonlinear model, the design procedure can be quite complex while the
claimed (global) stability can be ensured. If the controllers are designed
based on the linear model, the design procedure can be much more simpler since
a lot of linear design techniques are available, however, the claimed (global)
stability of the practical closed-loop system may not be maintained. As usual,
a trade-off thus exists.

If we are interested in the design of spacecraft attitude stabilizing
controllers based on the linearized model, it would be very helpful if we can
know some specific properties of the linearized model. This is particularly
the case if we are considering the constraints existing in the actuators since
the existence of a controller is highly dependent on the stability properties
of the open-loop system (see, for example, \cite{lin93}, \cite{ssy94tac},
\cite{zdl08tac}, and \cite{zhou15auto}). Hence, in this note, we are
interested in the stability and stabilization of the linearized spacecraft
attitude control system. Particularly, we will provide conditions on the
spacecraft parameters under which the linearized model is respectively
polynomially stable and is stable in the Lyapunov sense. We notice that in
these two cases the linearized spacecraft attitude control system can be
respectively stabilized semi-globally and globally by saturated linear
feedback (see \cite{zdl08tac} and \cite{zhou15auto}).

\section{Main Results}

We consider the following linear system%
\begin{equation}
\dot{\chi}=A\chi+Bu,\label{macs}%
\end{equation}
where $A$ and $B$ are given by \cite{psiaki01jgcd}, \cite{rh11icrast},
\cite{wertz78book}%
\begin{equation}
A=\left[
\begin{array}
[c]{cccccc}%
0 & 0 & 0 & 1 & 0 & 0\\
0 & 0 & 0 & 0 & 1 & 0\\
0 & 0 & 0 & 0 & 0 & 1\\
-4\omega_{0}^{2}\sigma_{1} & 0 & 0 & 0 & 0 & \omega_{0}\left(  1-\sigma
_{1}\right)  \\
0 & -3\omega_{0}^{2}\sigma_{2} & 0 & 0 & 0 & 0\\
0 & 0 & -\omega_{0}^{2}\sigma_{3} & \omega_{0}\left(  \sigma_{3}-1\right)   &
0 & 0
\end{array}
\right]  ,B=\left[
\begin{array}
[c]{ccc}%
0 & 0 & 0\\
0 & 0 & 0\\
0 & 0 & 0\\
\frac{1}{J_{x}} & 0 & 0\\
0 & \frac{1}{J_{y}} & 0\\
0 & 0 & \frac{1}{J_{z}}%
\end{array}
\right]  .\label{BC}%
\end{equation}
This linear system is the linearized model of the spacecraft attitude control
system. For detailed derivation of this equation, see \cite{psiaki01jgcd},
\cite{rh11icrast}, \cite{wertz78book}, \cite{zhou15auto}, and the references
therein. In the above equation, the symbol $\omega_{0}=\sqrt{\frac{\mu}{r^{3}%
}}$ denotes the orbital rate, in which $\mu=3.986\times10^{14}\mathrm{m}%
^{3}/\mathrm{s}^{2}$ is the gravity constant and $r$ is the semimajor axis of
the orbit, $J_{x},J_{y}$ and $J_{z}$ are the inertia of the spacecraft, $\chi$
contains the attitude angles and velocities, $u$ is the forces applied on the
spacecraft, and
\begin{equation}
\sigma_{1}=\frac{J_{y}-J_{z}}{J_{x}},\; \sigma_{2}=\frac{J_{x}-J_{z}}{J_{y}%
},\; \sigma_{3}=\frac{J_{y}-J_{x}}{J_{z}}.\label{eqs123}%
\end{equation}
In this note, we are interested in the stability properties of system
(\ref{macs}).

\begin{remark}
We notice that $\sigma_{i},i=1,2,3$ are dependent. In fact, by letting
$\frac{J_{x}}{J_{y}}=\beta_{1},\frac{J_{y}}{J_{z}}=\beta_{2},$ then%
\begin{equation}
\sigma_{1}=\frac{1}{\beta_{1}}-\frac{1}{\beta_{1}\beta_{2}},\; \sigma_{2}%
=\beta_{1}-\frac{1}{\beta_{2}},\; \sigma_{3}=\beta_{2}-\beta_{1}\beta
_{2},\label{eqadd1}%
\end{equation}
which implies that $\sigma_{i},i=1,2,3$ can be characterized by the two
scalars $\beta_{i},i=1,2.$
\end{remark}

Consider the following nonsingular matrices%
\begin{equation}
H=\left[
\begin{array}
[c]{cccccc}%
0 & 1 & 0 & 0 & 0 & 0\\
0 & 0 & 0 & 0 & \frac{1}{\omega_{0}} & 0\\
0 & 0 & 1 & 0 & 0 & 0\\
0 & 0 & 0 & \frac{1}{\omega_{0}} & 0 & 0\\
1 & 0 & 0 & 0 & 0 & 0\\
0 & 0 & 0 & 0 & 0 & \frac{1}{\omega_{0}}%
\end{array}
\right]  ,\;L=\left[
\begin{array}
[c]{ccc}%
0 & 1 & 0\\
1 & 0 & 0\\
0 & 0 & 1
\end{array}
\right]  \omega_{0}^{2},\label{eqH}%
\end{equation}
then it can be verified that
\begin{equation}
\left \{
\begin{array}
[c]{rl}%
HAH^{-1} & =\omega_{0}\left[
\begin{array}
[c]{ccc}%
A_{2} & 0 & 0\\
0 & 0 & A_{1}\\
0 & A_{3} & 0
\end{array}
\right]  \triangleq \omega_{0}\left[
\begin{array}
[c]{cc}%
A_{2} & 0\\
0 & A_{13}%
\end{array}
\right]  \triangleq \omega_{0}A_{0},\\
HBL & =\omega_{0}\left[
\begin{array}
[c]{ccc}%
B_{2} & 0 & 0\\
0 & B_{1} & 0\\
0 & 0 & B_{3}%
\end{array}
\right]  \triangleq \omega_{0}\left[
\begin{array}
[c]{cc}%
B_{2} & 0\\
0 & B_{13}%
\end{array}
\right]  \triangleq \omega_{0}B_{0},
\end{array}
\right.  \label{eqab}%
\end{equation}
where $\left(  A_{1},A_{2},A_{3}\right)  $ and $\left(  B_{1},B_{2}%
,B_{3}\right)  $ are independent of $\omega_{0}$ and are given by%
\begin{equation}
\left \{
\begin{array}
[c]{ll}%
A_{1}=\left[
\begin{array}
[c]{cc}%
0 & 1\\
-4\sigma_{1} & 1-\sigma_{1}%
\end{array}
\right]  ,\; & B_{1}=\left[
\begin{array}
[c]{c}%
0\\
\frac{1}{J_{x}}%
\end{array}
\right]  ,\\
A_{2}=\left[
\begin{array}
[c]{cc}%
0 & 1\\
-3\sigma_{2} & 0
\end{array}
\right]  ,\; & B_{2}=\left[
\begin{array}
[c]{c}%
0\\
\frac{1}{J_{y}}%
\end{array}
\right]  ,\\
A_{3}=\left[
\begin{array}
[c]{cc}%
0 & 1\\
-\sigma_{3} & \sigma_{3}-1
\end{array}
\right]  ,\; & B_{3}=\left[
\begin{array}
[c]{c}%
0\\
\frac{1}{J_{z}}%
\end{array}
\right]  .
\end{array}
\right.  \label{eqA123}%
\end{equation}

Our first result provides necessary and sufficient conditions to guarantee
that $A$ is polynomially stable, namely, $\operatorname{Re}\{s\} \leq0,\forall
s\in \lambda \left(  A\right)  ,$ the eigenvalue set of $A.$

\begin{proposition}
\label{claim0}All the eigenvalues of $A$ have non-positive real parts if and
only if%
\begin{equation}
\left \{
\begin{array}
[c]{l}%
0\leq \sigma_{2},\\
0\leq \sigma_{1}\sigma_{3}\triangleq \phi_{1},\\
0\leq3\sigma_{1}+\sigma_{3}\sigma_{1}+1\triangleq \phi_{2},\\
0\leq \left(  3\sigma_{1}+1+\sigma_{3}\sigma_{1}\right)  ^{2}-16\sigma
_{1}\sigma_{3}=\phi_{2}^{2}-16\phi_{1}.
\end{array}
\right.  \label{cond}%
\end{equation}
Moreover, if all the above conditions are satisfied, the eigenvalues
$s_{i},i\in \mathbf{I}\left[  1,6\right]  $ of $A$ are given by%
\begin{equation}
\left \{
\begin{array}
[c]{l}%
s_{1,2}=\pm \sqrt{3\sigma_{2}}\omega_{0}\mathrm{i},\\
s_{3,4}=\pm \sqrt{\frac{\phi_{2}+\sqrt{\phi_{2}^{2}-16\phi_{1}}}{2}}\omega
_{0}\mathrm{i},\\
s_{5,6}=\pm \sqrt{\frac{\phi_{2}-\sqrt{\phi_{2}^{2}-16\phi_{1}}}{2}}\omega
_{0}\mathrm{i}.
\end{array}
\right.  \label{eig}%
\end{equation}

\end{proposition}

\begin{proof}
Clearly, any $s\in \lambda \left(  A_{2}\right)  $ has non-positive real parts
if and only if $\sigma_{2}\geq0.$ On the other hand, if $s\in \lambda \left(
A_{13}\right)  ,$ then it follows from%
\begin{equation}
\det \left(  sI_{4}-A_{13}\right)  =\det \left[
\begin{array}
[c]{cc}%
sI_{2} & -A_{1}\\
-A_{3} & sI_{2}%
\end{array}
\right]  =\det \left(  s^{2}I_{2}-A_{1}A_{3}\right)  ,\label{eq78}%
\end{equation}
that $-s\in \lambda \left(  A_{13}\right)  .$ Hence $\operatorname{Re}%
\{s\} \leq0$ if and only if $\operatorname{Re}\{s\}=0.$ This is equivalent to
that all the eigenvalues of $A_{1}A_{3}$ are non-positive real numbers. Let%
\begin{equation}
f\left(  \lambda \right)  =\det \left(  \lambda I_{2}-A_{1}A_{3}\right)
=\lambda^{2}+\phi_{2}\lambda+4\phi_{1}.\label{eqadd2}%
\end{equation}
Hence we must have $f\left(  0\right)  \geq0,-\frac{\phi_{2}}{2}\leq
0,\Delta=\phi_{2}^{2}-16\phi_{1}\geq0,$ which are equivalent to (\ref{cond}).
Finally, if (\ref{cond}) are satisfied, the eigenvalues of $A$ can be easily
computed by noting (\ref{eq78}). The proof is finished.
\end{proof}

Under the condition that all the eigenvalues of $A$ have non-positive real
parts, the linear system (\ref{macs}) can be semi-globally stabilized by
linear feedback in the presence of actuator saturation (see, for example,
\cite{lin93} and \cite{zld11tac}). Particularly, explicit semi-globally
stabilizing controllers can be designed by the parametric Lyapunov equation
approach in \cite{zdl08tac}.

We next provide necessary and sufficient conditions to guarantee that $A$ is
Lyapunov stable (or neutrally stable), namely, all the eigenvalues of $A$ have
non-positive real parts and those eigenvalues on the imaginary axis are simple
(or equivalently, any eigenvalue on the imaginary axis has the same algebraic
and geometric multiplicities).

\begin{theorem}
\label{pp6}Let $A$ be given by (\ref{BC}). Then $A$ is Lyapunov stable if and
only if $\left(  \sigma_{1},\sigma_{2},\sigma_{3}\right)  $ satisfies
\begin{equation}
\left \{
\begin{array}
[c]{l}%
0<\sigma_{2},\\
0<\phi_{1},\\
0<\phi_{2},\\
0<\phi_{2}^{2}-16\phi_{1}.
\end{array}
\right.  \label{cond0}%
\end{equation}

\end{theorem}

\begin{proof}
Clearly, $A$ is Lyapunov stable if and only if $A_{2}$ and $A_{13}$ are all
Lyapunov stable. Moreover, the matrix $A_{2}$ is Lyapunov stable if and only
if $\sigma_{2}>0$ since $A_{2}$ has repeated eigenvalues 0 and is nonzero when
$\sigma_{2}=0.$

Obviously, if $\phi_{i}>0,i=1,2,3,$ then $A_{12}$ has four different
eigenvalues on the imaginary axis and is thus Lyapunov stable. Hence, we only
need to show the converse, namely, if $A_{13}$ is Lyapunov stable, then
$\phi_{i}>0,i=1,2,3.$

We first show that the Lyapunov stability of $A_{13}$ implies $\phi_{1}>0.$ We
show this by contradiction. Let $\phi_{1}=0.$ If $\sigma_{1}=\sigma_{3}=0,$
then we must have $\sigma_{2}=0$ (see equation (\ref{eqs123}))$,$ which is not
allowed. Hence we have either $\left(  \sigma_{1}=0,\sigma_{3}\neq0\right)  $
or $\left(  \sigma_{3}=0,\sigma_{1}\neq0\right)  .$ If $\left(  \sigma
_{1}=0,\sigma_{3}\neq0\right)  ,$ then $A_{13}$ has repeated eigenvalues $0$
with geometric multiplicity 1, and hence is not Lyapunov stable. If $\left(
\sigma_{3}=0,\sigma_{1}\neq0\right)  ,$ we consider two cases. Case 1:
$\sigma_{1}\neq-\frac{1}{3}.$ Then the algebraic and geometric multiplicities
of 0 are respectively 2 and 1, which implies that $A_{13}$ is not Lyapunov
stable. Case 2: $\sigma_{1}=-\frac{1}{3}.$ Then the algebraic and geometric
multiplicities of 0 are respectively 4 and 1, which, again, implies that
$A_{13}$ is not Lyapunov stable. In conclusion, we must have $\phi_{1}>0.$

We next show that the Lyapunov stability of $A_{13}$ implies $\phi_{2}>0.$
Otherwise, if $\phi_{2}=0,$ then $\phi_{2}^{2}-16\phi_{1}\geq0$ is equivalent
to $\phi_{1}\leq0,$ which contradicts with $\phi_{1}>0.$ Hence we must have
$\phi_{2}>0.$

We finally show that the Lyapunov stability of $A_{13}$ implies $\phi_{2}%
^{2}-16\phi_{1}>0.$ Otherwise, if $\phi_{2}^{2}-16\phi_{1}=0,$ then $A_{13}$
has repeated eigenvalues $\lambda_{3,4}=\pm \sqrt{\phi_{2}/2}\mathrm{i.}$
Notice that, for any $\lambda=\lambda_{i}\neq0,i=3,4,$ we have%
\begin{align}
\mathrm{rank}\left[
\begin{array}
[c]{cc}%
\lambda I_{2} & -A_{1}\\
-A_{3} & \lambda I_{2}%
\end{array}
\right]   &  =\mathrm{rank}\left[
\begin{array}
[c]{cc}%
I_{2} & 0\\
A_{3} & \lambda I_{2}%
\end{array}
\right]  \left[
\begin{array}
[c]{cc}%
\lambda I_{2} & -A_{1}\\
-A_{3} & \lambda I_{2}%
\end{array}
\right]  \nonumber \\
&  =\mathrm{rank}\left[
\begin{array}
[c]{cc}%
\lambda I_{2} & -A_{1}\\
0 & \lambda^{2}I_{2}-A_{3}A_{1}%
\end{array}
\right]  \nonumber \\
&  =2+\mathrm{rank}\left[  \lambda^{2}I_{2}-A_{3}A_{1}\right]  \nonumber \\
&  \geq3,\label{eqadd3}%
\end{align}
where we have noticed that the matrix
\begin{equation}
\lambda^{2}I_{2}-A_{3}A_{1}=\left[
\begin{array}
[c]{cc}%
\left(  \frac{5}{2}-\frac{1}{2}\sigma_{3}\right)  \sigma_{1}-\frac{1}{2} &
\sigma_{1}-1\\
4\left(  \sigma_{3}-1\right)  \sigma_{1} & \frac{1}{2}+\left(  \frac{1}%
{2}\sigma_{3}-\frac{5}{2}\right)  \sigma_{1}%
\end{array}
\right]  ,\label{eqadd4}%
\end{equation}
cannot be zero. Hence the geometric multiplicities of $\lambda_{i},i=3,4,$ is
1 and $A_{13}$ cannot be Lyapunov stable. The proof is finished.
\end{proof}

Consider the following Lyapunov equation%
\begin{equation}
A^{\mathrm{T}}P+PA=-D^{\mathrm{T}}D,\label{lya}%
\end{equation}
where $D$ is any matrix with proper dimensions. It is well known that the
above Lyapunov equation has a solution $P>0$ for some $D$ if and only if $A$
is neutrally stable (Lyapunov stable). Then we have the following corollary.

\begin{corollary}
\label{coro1}There exists a $P>0$ such that the Lyapunov equation (\ref{lya})
is satisfied for some $D$ if and only if $D=0$ and (\ref{cond0}) are satisfied.
\end{corollary}

\begin{proof}
It follows from Theorem \ref{pp6} that $A$ is neutrally stable if and only if
(\ref{cond0}) are satisfied and, in this case, all the eigenvalues of $A$ are
on the imaginary axis. We next show that $D$ must be zero. Let $y_{i}%
,i\in \mathbf{I}\left[  1,6\right]  $ be the eigenvectors of $A$ corresponding
to the eigenvalue $s_{i}.$ Assume that $D$ is not zero. Then there exists a
$y_{j}$ such that $Dy_{j}\neq0$ (otherwise, we must have $D=0$). Consequently,
it follows from (\ref{lya}) that%
\begin{align}
-\left \Vert Dy_{j}\right \Vert ^{2} &  =-y_{j}^{\mathrm{H}}D^{\mathrm{T}}%
Dy_{j}\nonumber \\
&  =y_{j}^{\mathrm{H}}\left(  A^{\mathrm{T}}P+PA\right)  y_{j}\nonumber \\
&  =2\operatorname{Re}\left(  s_{j}\right)  y_{j}^{\mathrm{H}}Py_{j}%
\nonumber \\
&  =0,\label{eqadd5}%
\end{align}
which contradicts with $Dy_{j}\neq0.$ The proof is finished.
\end{proof}

We can actually give explicit solutions to the Lyapunov equation (\ref{lya}),
as shown in the following theorem.

\begin{theorem}
Assume that $\sigma_{1}\sigma_{2}\sigma_{3}\neq0.$ Then all the solutions to
the Lyapunov equation (\ref{lya}) with $D=0$ are given by%
\begin{equation}
P=H^{\mathrm{T}}P_{0}H,\;P_{0}=\left[
\begin{array}
[c]{ccc}%
P_{2} &  & \\
& P_{1} & \\
&  & P_{3}%
\end{array}
\right]  ,\label{eqP}%
\end{equation}
in which $P_{2}=\mathrm{diag}\{3\sigma_{2}\alpha_{2},\alpha_{2}\}$ with
$\alpha_{2}$ being any constants, and%
\begin{equation}
\left \{
\begin{array}
[c]{l}%
P_{1}=\left[
\begin{array}
[c]{cc}%
\sigma_{3}\left(  \alpha_{3}+\left(  1-\sigma_{1}\right)  \alpha_{13}\right)
& -\sigma_{3}\alpha_{13}\\
-\sigma_{3}\alpha_{13} & \alpha_{1}%
\end{array}
\right]  ,\\
P_{3}=\left[
\begin{array}
[c]{cc}%
4\sigma_{1}\left(  \alpha_{1}+\left(  1-\sigma_{3}\right)  \alpha_{13}\right)
& 4\sigma_{1}\alpha_{13}\\
4\sigma_{1}\alpha_{13} & \alpha_{3}%
\end{array}
\right]  ,
\end{array}
\right.  \label{eqp3}%
\end{equation}
where $\alpha_{1},\alpha_{3}$ and $\alpha_{13}$ are any scalars and such that%
\begin{equation}
\left(  1-\sigma_{1}\right)  \left(  \alpha_{1}-\frac{J_{x}}{J_{z}}\alpha
_{3}\right)  +\left(  4\sigma_{1}-\sigma_{3}\right)  \alpha_{13}=0.\label{eq1}%
\end{equation}
Particularly, if we choose $\alpha_{13}=0$ and $\alpha_{1}=\frac{J_{x}}{J_{z}%
}\alpha_{3},$ then we have
\begin{equation}
P_{1}=\mathrm{diag}\left \{  \sigma_{3}\alpha_{3},\frac{J_{x}}{J_{z}}\alpha
_{3}\right \}  ,\;P_{3}=\mathrm{diag}\left \{  4\frac{J_{x}}{J_{z}}\sigma
_{1}\alpha_{3},\alpha_{3}\right \}  .\label{eqp12}%
\end{equation}

\end{theorem}

\begin{proof}
We only need to show that all the solutions $P_{0}$ to the equation
$A_{0}^{\mathrm{T}}P_{0}+P_{0}A_{0}=0$ are in the form of (\ref{eqP}). Since
$A_{i},i=1,2,3,$ are all nonsingular, by noting the structure of $A_{0},$ we
know that $P_{0}$ takes the following form
\begin{equation}
P_{0}=\left[
\begin{array}
[c]{ccc}%
P_{2} & 0 & 0\\
0 & P_{1} & P_{13}\\
0 & P_{13} & P_{3}%
\end{array}
\right]  ,\label{eqp0}%
\end{equation}
which gives the following five equations%
\begin{equation}
\left \{
\begin{array}
[c]{l}%
0=A_{2}^{\mathrm{T}}P_{2}+P_{2}A_{2},\\
0=A_{3}^{\mathrm{T}}P_{13}^{\mathrm{T}}+P_{13}A_{3},\\
0=A_{1}^{\mathrm{T}}P_{13}+P_{13}^{\mathrm{T}}A_{1},\\
0=A_{3}^{\mathrm{T}}P_{3}+P_{1}A_{1},\\
0=A_{1}^{\mathrm{T}}P_{1}+P_{3}A_{3}.
\end{array}
\right.  \label{eq908}%
\end{equation}
We first consider the first equation. Let%
\begin{equation}
P_{2}=\left[
\begin{array}
[c]{cc}%
p_{21} & p_{22}\\
p_{22} & p_{23}%
\end{array}
\right]  ,\label{eqadd7}%
\end{equation}
where the parameters are to be specified. Then%
\begin{equation}
A_{2}^{\mathrm{T}}P_{2}+P_{2}A_{2}=\left[
\begin{array}
[c]{cc}%
-6\sigma_{2}p_{22} & p_{21}-3\sigma_{2}p_{23}\\
p_{21}-3\sigma_{2}p_{23} & 2p_{22}%
\end{array}
\right]  .\label{eqadd8}%
\end{equation}
Hence we must have $p_{22}=0$ and $p_{21}=3\sigma_{2}p_{23}.$ This is just the
desired result if we set $p_{23}=\alpha_{2}.$with $a_{2}$ being any constants.

We assume that $P_{13}$ takes the form%
\begin{equation}
P_{13}=\left[
\begin{array}
[c]{cc}%
p_{131} & p_{132}\\
p_{133} & p_{134}%
\end{array}
\right]  ,\label{eqadd9}%
\end{equation}
where the parameters are to be specified. Then we can compute%
\begin{align}
A_{3}^{\mathrm{T}}P_{13}^{\mathrm{T}}+P_{13}A_{3} &  =\left[
\begin{array}
[c]{cc}%
-2\sigma_{3}p_{132} & p_{131}+\left(  \sigma_{3}-1\right)  p_{132}-\sigma
_{3}p_{134}\\
p_{131}+\left(  \sigma_{3}-1\right)  p_{132}-\sigma_{3}p_{134} &
2p_{133}+2\left(  \sigma_{3}-1\right)  p_{134}%
\end{array}
\right]  ,\label{eqadd10}\\
A_{1}^{\mathrm{T}}P_{13}+P_{13}^{\mathrm{T}}A_{1} &  =\left[
\begin{array}
[c]{cc}%
-8\sigma_{1}p_{133} & p_{131}+\left(  1-\sigma_{1}\right)  p_{133}-4\sigma
_{1}p_{134}\\
p_{131}+\left(  1-\sigma_{1}\right)  p_{133}-4\sigma_{1}p_{134} &
2p_{132}+2\left(  1-\sigma_{1}\right)  p_{134}%
\end{array}
\right]  .\label{eqadd11}%
\end{align}
Hence we must have $p_{132}=0,p_{133}=0$ and $p_{131}=\sigma_{3}p_{134}.$
Consequently, we get%
\begin{align}
A_{3}^{\mathrm{T}}P_{13}^{\mathrm{T}}+P_{13}A_{3} &  =\left[
\begin{array}
[c]{cc}%
0 & 0\\
0 & 2\left(  \sigma_{3}-1\right)
\end{array}
\right]  p_{134},\label{eqadd12}\\
A_{1}^{\mathrm{T}}P_{13}+P_{13}^{\mathrm{T}}A_{1} &  =\left[
\begin{array}
[c]{cc}%
0 & \sigma_{3}-4\sigma_{1}\\
\sigma_{3}-4\sigma_{1} & 2\left(  1-\sigma_{1}\right)
\end{array}
\right]  p_{134}.\label{eqadd13}%
\end{align}
We consider two cases. Case 1: $\sigma_{1}\neq1.$ In this case we must have
$p_{134}=0.$ Case 2: $\sigma_{1}=1.$ Then we have $\left(  1-\sigma
_{3}\right)  p_{134}=0$ and $\left(  \sigma_{3}-4\right)  p_{134}=0.$ Since
$\sigma_{3}$ can not equal to 1 and 4 simultaneously, we also must have
$p_{134}=0.$ As a result, we have $P_{13}=0.$

We next let $P_{1}$ and $P_{3}$ take the form%
\begin{equation}
P_{1}=\left[
\begin{array}
[c]{cc}%
p_{11} & p_{12}\\
p_{12} & p_{13}%
\end{array}
\right]  ,P_{3}=\left[
\begin{array}
[c]{cc}%
p_{31} & p_{32}\\
p_{32} & p_{33}%
\end{array}
\right]  ,\label{eqadd14}%
\end{equation}
where the parameters are to be specified. Then%
\begin{align}
A_{3}^{\mathrm{T}}P_{3}+P_{1}A_{1} &  =\left[
\begin{array}
[c]{cc}%
-\sigma_{3}p_{32}-4\sigma_{1}p_{12} & -\sigma_{3}p_{33}+p_{11}+p_{12}%
-\sigma_{1}p_{12}\\
-p_{32}+\sigma_{3}p_{32}-4p_{13}\sigma_{1} & p_{32}-p_{33}+\sigma_{3}%
p_{33}+p_{12}+p_{13}-p_{13}\sigma_{1}%
\end{array}
\right]  ,\label{eqadd15}\\
A_{1}^{\mathrm{T}}P_{1}+P_{3}A_{3} &  =\left[
\begin{array}
[c]{cc}%
-\sigma_{3}p_{32}-4\sigma_{1}p_{12} & p_{31}-p_{32}+\sigma_{3}p_{32}%
-4p_{13}\sigma_{1}\\
-\sigma_{3}p_{33}+p_{11}+p_{12}-\sigma_{1}p_{12} & p_{32}-p_{33}+\sigma
_{3}p_{33}+p_{12}+p_{13}-p_{13}\sigma_{1}%
\end{array}
\right]  .\label{eqadd16}%
\end{align}
Hence we must have
\begin{equation}
\left \{
\begin{array}
[c]{l}%
p_{12}=-\sigma_{3}p_{32}/4/\sigma_{1},\\
p_{11}=-(-\sigma_{3}p_{33}+p_{12}-\sigma_{1}p_{12}),\\
p_{31}=-(-p_{32}+\sigma_{3}p_{32}-4p_{13}\sigma_{1}).
\end{array}
\right.  \label{eqadd17}%
\end{equation}
Consequently, the above two equations are satisfied if and only if%
\begin{equation}
\left(  1-\sigma_{1}\right)  p_{13}+\left(  1-\frac{\sigma_{3}}{4\sigma_{1}%
}\right)  p_{32}+\left(  \sigma_{3}-1\right)  p_{33}=0,\label{eqadd18}%
\end{equation}
which is
\begin{equation}
\left(  1-\sigma_{1}\right)  p_{13}+\left(  4\sigma_{1}-\sigma_{3}\right)
p_{32}+\left(  \sigma_{3}-1\right)  p_{33}=0.\label{eqadd19}%
\end{equation}
by setting $p_{32}\rightarrow4p_{32}\sigma_{1}.$ This is just (\ref{eq1})
since
\begin{equation}
\sigma_{3}-1=\frac{J_{x}}{J_{z}}\left(  \sigma_{1}-1\right)  .\label{eqadd20}%
\end{equation}
and we have denoted $p_{13}=\alpha_{1},p_{33}=\alpha_{3}$ and $p_{32}%
=\alpha_{12}$. The proof is finished by noting that $P_{1}$ and $P_{3}$ are
now exactly in the form of (\ref{eqp3}).
\end{proof}

Of course, the parameters $\alpha_{1},\alpha_{2},\alpha_{3}$ and $\alpha_{13}$
can be well chosen such that $P_{i}>0,i=1,2,3$, if the conditions in
(\ref{cond0}) are satisfied.

The positive definite solutions to the Lyapunov equation (\ref{lya}) can be
utilized to stabilize system (\ref{macs}) globally in the presence if actuator
saturation even when $B\left(  t\right)  $ is a time-varying periodic matrix.
Interested readers may refer to \cite{zhou15auto} for details.

\section{Conclusion}

This note has studied the stability properties of the linearized spacecraft
attitude control system. Necessary and sufficient conditions are respectively
provided to guarantee that the considered systems are Lyapunov stable
(neutrally stable) and polynomially stable.

\end{document}